\documentclass[1p,number,sort&compress]{elsarticle}
\usepackage{amsmath, mathrsfs, amssymb, enumerate, txfonts, color}
 
\numberwithin{equation}{section}
\allowdisplaybreaks
\newtheorem{theorem}{Theorem}[section]
\newtheorem{lemma}[theorem]{Lemma}
\newtheorem{cor}[theorem]{Corollary}

\newdefinition{definition}[theorem]{Definition}
\newdefinition{remark}[theorem]{Remark}
\newdefinition{example}[theorem]{Example}
\newproof{proof}{Proof}
\newcommand{\wconv}{\rightharpoonup}
\newcommand{\wsconv}{\mathrel{\vbox{\offinterlineskip\ialign{\hfil##\hfil\cr $\hspace{-0.4ex}\textnormal{\scriptsize{*}}$\cr \noalign{\kern-0.6ex}  $\rightharpoonup$\cr}}}}
\newcommand*\diff{\mathop{}\!\mathrm{d}}
\begin{document}
\begin{frontmatter}
\title{Nonlinear evolution equations with\\
exponentially decaying memory:\\
Existence via time discretisation, uniqueness, and stability\footnote{This work has been supported by Deutsche Forschungsgemeinschaft through Collaborative Research Center 910
``Control of self-organizing nonlinear systems: Theoretical methods and concepts of application''.
}}
\author{Andr\'e Eikmeier}
\ead{eikmeier@math.tu-berlin.de}
\author{Etienne Emmrich\corref{ee}}
\cortext[ee]{Corresponding author}
\ead{emmrich@math.tu-berlin.de}
\author{Hans-Christian Kreusler}
\address{Technische Universit{\"a}t Berlin, Institut f\"ur Mathematik, Stra{\ss}e des 17.~Juni 136, 10623 Berlin, Germany}
\ead{kreusler@math.tu-berlin.de}

\begin{abstract}
The initial value problem for an evolution equation of type $v' + Av + BKv = f$ is studied, where $A:V_A \to V_A'$ is a monotone, coercive operator and where $B:V_B \to V_B'$ induces an inner product. The Banach space $V_A$ is not required to be embedded in $V_B$ or vice versa.
The operator $K$ incorporates a Volterra integral operator in time of convolution type with an exponentially decaying kernel. Existence of a global-in-time solution is shown by proving convergence of a suitable time discretisation. Moreover, uniqueness as well as stability results are proved. Appropriate integration-by-parts formulae are a key ingredient for the analysis.
\end{abstract}
\end{frontmatter}
\section{Introduction}
\subsection{Problem statement and main result}

\noindent
We are concerned with initial value problems of the type
\begin{subequations}
\label{abstract1}
\begin{align}\label{abstract1a}
v' + Av + BKv &= f \quad\text{in } (0,T) , \\
v(0) &= v_0 ,
\end{align}
where
\begin{equation}\label{Volterra}
(Kv)(t) = u_0 + \int_0^t k(t-s) v(s) \diff s , \quad
k(z) = \lambda e^{-\lambda z} .
\end{equation}
\end{subequations}
Here, $T>0$ is the time under consideration, $\lambda > 0$ is a given parameter, and $u_0$, $v_0$, $f$ are given data of the problem.

The operator $A : V_A \to V_A'$ is  a hemicontinuous, monotone, coercive operator satisfying a certain growth condition, where $V_A$ denotes a real, reflexive Banach space. The operator $B : V_B \to V_B'$
is a linear, symmetric, bounded, strongly positive operator
on a real Hilbert space $V_B$. We assume that both $V_A$ and $V_B$ are densely and continuously embedded in a real Hilbert space $H$, which is identified with its dual. We emphasise that we do not require $V_A$ to be embedded in $V_B$ (or vice versa) but only assume that
$V=V_A\cap V_B$ is densely embedded in $V_A$ as well as in $V_B$. This yields the scale
\begin{equation}\label{embed}
V_A\cap V_B=V\subset V_C \subset H = H' \subset V_C'\subset V'= V_A'+V_B',\quad C\in \{A,B\},
\end{equation}
of Banach and Hilbert spaces with dense and continuous embeddings. Moreover, we assume that $V$ is separable.

Since the kernel $k$ in (\ref{Volterra}) is assumed to be of exponential type, (\ref{abstract1a}) can easily be derived from the system
\begin{subequations}\label{system}
\begin{align}
v' + Av + Cu &= f , \\
(u-u_0)' + \lambda (u-u_0)  &=\lambda Dv .
\end{align}
\end{subequations}
with suitable linear operators $C, D$ such that $B=CD$ (and with $u_0$ being replaced by $Du_0$).

Such systems appear, e.g., in the description of viscoelastic fluid
flow (see, e.g., Desch, Grimmer, and Schappacher~\cite{Desch88}, MacCamy~\cite{MacCamy77b}) as well as heat flow in materials with memory (see, e.g., MacCamy~\cite{MacCamy77a}, Miller~\cite{Miller78}).
%
%
%

Other applications arise in, e.g., non-Fickian diffusion models describing diffusion processes of a penetrant through a viscoelastic material (see, e.g., Edwards~\cite{Edwards96}, Edwards and Cohen~\cite{EdwardsCohen95}, Shaw and Whiteman~\cite{ShawWhiteman98}), where, apart from the usual diffusion rate of the penetrant, the change in the internal structure of the viscoelastic material has to be considered. Non-Fickian diffusion also plays a role in, e.g., mathematical biology (see, e.g., Cushing~\cite{Cushing}, Fedotov and Iomin~\cite{FedotovIomin08}, Mehrabian and Abousleiman~\cite{MehrabianAbousleiman11}).

Note that all the results obtained in this paper are also valid for kernels of the type $k(z) = c e^{- \lambda z}$ ($ c, \lambda > 0 $). The proof only differs for the stability with respect to perturbations of the kernel. However, in many applications, kernels of the type given in (\ref{Volterra}) naturally arise.
In such applications, one often deals with a coupled system of the type (\ref{system}), where $1/\lambda$ can  be interpreted as a relaxation or averaged delay time. If $\lambda$ tends to $0$ then the system decouples with $u(t) \equiv u_0$. The limit $\lambda\to \infty$ results in a first order equation for $v$  without memory.
In contrast to this, in the case where $k(z)=c e^{- \lambda z}$, the limit $\lambda\to 0$ results in an evolution equation for $u$ of second order in time (see Emmrich and Thalhammer~\cite{EmmrichThalhammer1}), whereas $\lambda \to \infty$ results again in a first order evolution equation for $v$.

Our goal in this paper is twofold: On the one hand, we wish to prove existence of generalised solutions to (\ref{abstract1}). On the other hand, we wish to prove convergence of a simple time discretisation method, which relies on the implicit Euler or Rothe method combined with a product quadrature. Moreover, we prove stability of the solution against perturbations of the problem data (including perturbations of $\lambda$), which also implies uniqueness.

\subsection{Literature overview}

\noindent
Most of the results on evolutionary Volterra integral equations available so far are, indeed, more general
with respect to the integral kernel (often dealing with memory of positive type). Our results, however, include the case of different domains of definition of the underlying operators involved without assuming that $V_A$ is embedded in $V_B$ (or vice versa). Such a situation occurs, e.g., if $A$ is a spatial differential operator of order lower than that of $B$.
This has, to the best knowledge of the authors, not yet been studied. Dealing with memory of exponentially decaying type can thus be seen as a first step within this more general functional analytic framework.

Well-posedness for linear evolutionary integral equations in Banach or Hilbert spaces has been studied in detail by many authors for a long time. We only refer to the standard monographs
Gripenberg et al.~\cite{Gripen} and Pr\"u{\ss}~\cite{Pruess}, and the references cited therein. For the semigroup approach, we also refer to
B\'{a}tkai and Piazzera \cite{Batkai}. For the finite dimensional case (including nonlinear problems), one may also consult Corduneanu~\cite{Corduneanu}.

Whereas there are many results available for linear problems, there is less known for the nonlinear case. Classical results for nonlinear problems are due to Cl\'{e}ment, Crandall, Dafermos, Desch, Gripenberg, Londen, MacCamy, Nohel, and others. One of the equations that is mostly studied is of the type
\begin{equation*}
v' + a\ast Av = f ,
\end{equation*}
where $a$ is a suitable kernel function (e.g., completely monotone or of positive type) and $A$ is a suitable nonlinear operator satisfying certain monotonicity assumptions, see, e.g., Dafermos and Nohel~\cite{DafermosNohel}, Londen~\cite{Londen78}, Gripenberg~\cite{Gripenberg78,Gripenberg79,Gripenberg94}, Webb~\cite{Webb76, Webb77, Webb78}.

For example, MacCamy and Wong~\cite{MacCamyWong} consider a class of nonlinear integro-differential equations with completely monotone kernel in a Hilbert space setting. Gajewski~et al.~\cite{GGZ} study rather general classes of nonlinear evolution equations with Volterra operators but are restricted to the Hilbert space case $V_A=V_B$.
Crandall et al.~\cite{CrandallLondenNohel} consider a doubly nonlinear problem governed by maximal monotone subdifferential operators, where the domain of definition of one of the operators is dense and continuously embedded into the domain of definition of the other operator, see also Crandall and Nohel~\cite{CrandallNohel}.
An interesting result for a class of doubly nonlinear integro-differential equations governed by an $m$-dissipative operator can be found in Grimmer and Zeman~\cite{GrimmerZeman}. The nonlocality in time here is of a special form incorporating this $m$-dissipative operator again.
The domain of definition of the principal part of the equations, however, equals the domain of definition of the nonlocality in time. The method of proof relies upon nonlinear semigroup theory.
More recent results can be found, e.g., in Barbu et al. \cite{Barbu00}, Grasselli and Lorenzi \cite{Grasselli}, Cannarsa and Sforza \cite{Cannarsa03}, Gilardi and Stefanelli \cite{Gilardi06, Gilardi07, Stefanelli04}, Zacher \cite{Zacher08,Zacher09} and Bonnacorsi and Desch \cite{Desch13}.
The above cited references do not cover the class of problems we consider here.

The question of numerical approximation, mostly for linear or semilinear problems with memory of positive type, has been dealt with by McLean, Sloan, Thom\'{e}e, Wahlbin (see, e.g., \cite{SloanThomee86, McLeanThomee93,Lubichetal96,LarssonThomeeWahlbin98}).
The focus is on Galerkin finite element methods combined with suitable time discretisation methods  based on the backward Euler scheme (see, e.g., Chen, Thom\'ee, Wahlbin \cite{Chen92}, Chen and Shih \cite{Chen98}).
Time discretisation methods have also been studied by Calvo, Lubich and Palencia (see, e.g., \cite{Calvo07,Cuesta04}) and, in particular,
convolution quadrature by Lubich \cite{Lubich88a, Lubich88b, Lubich04} and Ferreira \cite{Ferreira13}.
Indeed, our proof of existence also relies upon the convergence of a numerical scheme that is based on a convolution quadrature.

\subsection{Organisation of the paper}

\noindent
The paper is organised as follows: The general notation is explained in Section~2. In Section~3, we state the main assumptions on the operators $A$ and $B$ and collect some preliminary results on their properties.
The main existence result (Theorem~\ref{MainThm}) is provided in
Section~4 by showing (weak) convergence of a suitable time discretisation to (\ref{abstract1}). The crucial point here is an integration-by-parts formula given in Lemma~\ref{PartInt}. In the last section, we show uniqueness and stability with respect to perturbations of the problem data. In particular, we also consider perturbations of the kernel parameter $\lambda$.

\section{Notation}

\noindent
For a Banach space $X$, we denote its norm by $\|\cdot\|_X$, its dual by $X'$, equipped with the standard norm $\|\cdot\|_{X'}$, and the duality pairing by $\langle\cdot, \cdot\rangle$. We recall that $X'$ is reflexive and separable if $X$ does so. For a Hilbert space $X$, we denote the inner product (inducing the norm $\|\cdot\|_X$) by $(\cdot,\cdot)_X$.
The intersection of two Banach spaces $X, Y$ is equipped with the norm $\|\cdot\|_{X\cap Y} = \|\cdot\|_X + \|\cdot\|_Y$, whereas the sum $X+Y$ is equipped with the norm
\begin{equation*}
\|g\|_{X+Y}= \inf\left\{ \max\left(\|g_X\|_X, \|g_Y\|_Y\right) : g = g_X + g_Y \text{ with } g_X \in X , \, g_Y \in Y\right\}\! .
\end{equation*}
We recall that $(X\cap Y)' = X' + Y'$, see also Gajewski~et~al.~\cite[pp.~12ff.]{GGZ}.

For a real, reflexive, separable Banach space $X$, the Bochner--Lebesgue spaces $L^r(0,T;X)$ ($r \in [1,\infty]$) are defined in the usual way and equipped with the standard norm.
Denoting by $r' = r/(r-1)$ the conjugate of $r\in (1,\infty)$ with $r'=\infty$ if $r = 1$, we have that $\left(L^r(0,T;X)\right)' = L^{r'}(0,T;X')$ if $r \in [1,\infty)$;
the duality pairing is given by
\begin{equation*}
\langle g , v \rangle = \int_0^T \langle g(t) , v(t) \rangle \diff t  ,
\end{equation*}
see, e.g., Diestel and Uhl~\cite[Thm.~1 on p.~98, Coroll.~13 on p.~76, Thm.~1 on p.~79]{DiestelUhl}.
Moreover, $L^r(0,T;X)$ is reflexive if $r\in (1,\infty)$ (see \cite[Coroll.~2 on p.~100]{DiestelUhl}) and $L^1(0,T;X)$ is separable.

By $W^{1,r}(0,T;X)$ ($r \in [1,\infty]$), we denote the Banach space of
functions $u\in L^r(0,T;X)$ whose distributional time derivative $u'$ is again in $L^r(0,T;X)$; the space is equipped with the standard norm. Note that if $u\in W^{1,1}(0,T;X)$ then $u$ equals almost everywhere
a function that is in $\mathscr{AC}([0,T];X)$, i.e., a function that is absolutely continuous on $[0,T]$ as a function taking values in $X$. Moreover, $W^{1,1}(0,T;X)$ is continuously
embedded in the Banach space $\mathscr{C}([0,T];X)$ of functions that are
continuous on $[0,T]$ as functions with values in $X$ (see, e.g.,
Roub\'{\i}\v{c}ek~\cite[Chapter~7]{Roubicek} for more details).
By $\mathscr{C}_w([0,T];X)$, we denote the space of functions that are continuous on $[0,T]$ with respect to the weak topology in $X$.

The space of infinitely many times differentiable real functions with compact support in $(0,T)$ is denoted by $\mathscr{C}_c^\infty(0,T)$. The space of on $[0,T]$ continuously differentiable real functions is denoted by $\mathscr{C}^1([0,T])$.
By $c$, we denote a generic positive constant. We set $\sum_{j=1}^0 x_j = 0$ for $x_j$ whatsoever.

\section{Main assumptions and preliminary results}\label{Sec2}

\noindent
Let $V_A$ be a real, reflexive Banach space and $V_B$ and $H$ be real Hilbert spaces satisfying (\ref{embed}). Moreover, we assume that $V=V_A \cap V_B$ is separable.

The structural properties we assume for the operators $A$ and $B$ throughout this paper are as follows:

\smallskip\noindent
{\bf Assumption~$(\mathbf{A})$}
\begin{em} The operator $A:V_A\to V_A'$ satisfies for all $u, v, w\in V_A$
\begin{itemize}\leftskip-0.5em
\item[\em (i)] $\theta \mapsto \langle A (u + \theta v) , w \rangle \in \mathscr{C}([0,1])$  \quad (hemicontinuity)\/,
\end{itemize}
and there
exist $p\in (2,\infty)$, $\mu_A, \beta_A > 0$, $c_A \ge 0$ such that for all $v, w\in V_A$
\begin{itemize}\leftskip-0.5em
\item[\em (ii)] $\langle Av - Aw, v - w \rangle \ge 0$ \quad (monotonicity)\/,
\item[\em (iii)] $\langle Av , v\rangle \ge \mu_A \, \|v\|_{V_A}^p  - c_A$ \quad ($p$-coercivity)\/,
\item[\em (iv)] $\|Av\|_{V_A'} \le \beta_A \left( 1 + \|v\|_{V_A}^{p-1} \right)$ \quad ($(p-1)$-growth)\/.
\end{itemize}
\end{em}

\smallskip\noindent
{\bf Assumption~$(\mathbf{B})$}
\begin{em}
The linear operator $B: V_B \to V_B'$ is symmetric and there exist
$\mu_B, \beta_B  > 0$ such that for all $v\in V_B$
\begin{itemize}\leftskip-0.5em
\item[\em (i)] $\langle Bv, v\rangle \ge \mu_B\, \|v\|_{V_B}^2$ \quad (strong positivity)\/,
\item[\em (ii)] $\|Bv\|_{V_B'} \le \beta_B \, \|v\|_{V_B}$ \quad (boundedness)\/.
\end{itemize}
\end{em}

By $\|\cdot\|_B = \langle B\cdot, \cdot\rangle^{1/2}$, we denote the norm induced by $B$,
which is equivalent to $\|\cdot\|_{V_B}$. Further, by $L^2(0,T;B)$, we denote the space $L^2(0,T;(V_B,\Vert \cdot\Vert_B))$.

We shall remark that the operators $A$ and $B$ can be extended, as usual, to
operators acting on functions defined on $[0,T]$ and taking values in $V_A$ and $V_B$,
respectively.
Since the operator $A:V_A\to V_A'$ is hemicontinuous and monotone, it is
also demicontinuous (see, e.g., Zeidler~\cite[Propos.~26.4 on
p.~555]{ZeidlerIIB}). Therefore, the operator $A$  maps, in view of the separability of $V_A'$ and the theorem of Pettis
(see, e.g., Diestel and Uhl~\cite[Thm.~2 on p.~42]{DiestelUhl}),
a Bochner measurable function $v: [0,T] \to V_A$ into a Bochner measurable
function $Av : [0,T] \to V_A'$ with $(Av)(t) = A v(t)$ ($t\in [0,T]$).
Because of the growth condition, the operator $A$ then maps $L^p(0,T;V_A)$ into $\left(L^p(0,T;V_A)\right)' = L^{p'}(0,T;V_A')$.
The linear, symmetric, bounded, strongly positive operator $B:V_B \to V_B'$ extends, via $(Bu)(t) = Bu(t)$ for a function $u: [0,T] \to V_B$,
to a linear, symmetric, bounded, strongly positive operator mapping
$L^2(0,T;V_B)$ into its dual. Indeed, $B$ can also be seen as a linear, bounded operator mapping $L^r(0,T;V_B)$ into $L^r(0,T;V_B')$ for any $r\in [1,\infty]$.

With respect to the operator $K$ defined by (\ref{Volterra}), we make the following observations.

\begin{lemma}\label{LemmaK}
Let $k$ be as in (\ref{Volterra}) and let $u_0 \in V_B$.
Then $K$ is an affine-linear, bounded mapping of $L^2(0,T;V_B)$ into itself with
\begin{equation*}
\|Kv-u_0\|_{L^2(0,T;V_B)} \le \|k\|_{L^1(0,T)} \|v\|_{L^2(0,T;V_B)}  , \quad v \in L^2(0,T;V_B) ,
\end{equation*}
where $\|k\|_{L^1(0,T)}= 1-e^{-\lambda T}$. Moreover, $K$
is an affine-linear, bounded mapping of $L^1(0,T;V_B)$ into $\mathscr{AC}([0,T];V_B)$ with
\begin{equation*}
\|Kv-u_0\|_{\mathscr{C}([0,T];V_B)} \le \lambda\, \|v\|_{L^1(0,T;V_B)}  , \quad v \in L^1(0,T;V_B) .
\end{equation*}
\end{lemma}

The standard proof is omitted here.


Note that in the assertions above, $V_B$ can be replaced by $H$.
An immediate consequence of the lemma above is that $BK : L^2(0,T;V_B) \to L^2(0,T;V_B')$ as well as $BK : L^1(0,T;V_B) \to L^\infty(0,T;V_B')$ is affine-linear and bounded if Assumption~$(\mathbf{B})$ is satisfied.

We will often make use of the following relation, which indeed is crucial within this work and reflects the exponential type of the memory kernel.
For a function $v\in L^1(0,T;X)$, $X$ being an arbitrary Banach space, we have that for almost all $t\in (0,T)$
\begin{equation}
 \label{GleichungK}
(Kv)'(t)=\lambda\, \Big( v(t) - \big(\!\left(Kv\right)(t) - u_0\big) \Big)  .
\end{equation}
If $X$ is a Hilbert space and if $v \in L^2(0,T;X)$, then testing this relation by $Kv - u_0$ immediately implies for all $t\in [0,T]$
\begin{equation}\label{Kpositiv}
\int_0^t ( (Kv)(s) - u_0 , v(s))_X \diff s =
\int_0^t \|(Kv)(s) - u_0\|_X^2 \diff s + \frac{1}{2\lambda}\, \|(Kv)(t) - u_0\|_X^2 ,
\end{equation}
which shows that the memory term is of positive type.

\section{Main result: existence via time discretisation}\label{Sec3}

\noindent
In this section, we show existence of generalised solutions to (\ref{abstract1}) by proving weak or weak* convergence of a sequence of approximate solutions constructed from a suitable time discretisation. We commence by studying the corresponding numerical scheme and its properties.

\subsection{Time discretisation}

\noindent
For $N\in \mathbb{N}$, let $\tau = T/N$ and $t_n = n\tau$
($n=0, 1, \dots , N$).
Let $v^0\approx v_0$, $u^0\approx u_0$ and $\{f^n\}_{n=1}^N \approx f$ be given
approximations of the problem data $v_0$, $u_0$ and $f$, respectively.
We look for approximations $v^n \approx v(t_n)$ ($n=1, 2, \dots , N$).

The numerical scheme we consider combines the implicit Euler method with a
convolution or product quadrature for the integral operator $K$ and reads
\begin{subequations}\label{scheme}
\begin{align}\label{scheme1}
\frac{1}{\tau} \left(v^n - v^{n-1}\right) + Av^n + \left(B K^\tau_{u^0} v\right)^n = f^n , \quad n= 1, 2, \dots , N ,
\end{align}
where
\begin{equation}\label{Ktau}
\left(K^\tau_{u^0} v\right)^n := u^0 + \tau \sum_{j=1}^n \gamma_{n-j+1} v^j .
\end{equation}
To be precise, $K^\tau_{u^0}$ acts on a grid function $\{v^n\}_{n=1}^N$ and, with a slight abuse of notation, the evaluation at $n$ is denoted by $\left(K^\tau_{u^0} v\right)^n$.
The coefficients $\gamma_i$ ($i=1, 2, \dots , N$) are given by
\begin{equation} \label{DefGamma}
\gamma_{i} =   \int_0^1 k\Big( (i - \hat{s})\tau\Big) \diff \hat{s},
\end{equation}
\end{subequations}
and hence
\begin{equation*}
\gamma_{n-j+1} =   \int_0^1 k\Big( (n-j+1 - \hat{s})\tau\Big) \diff \hat{s}
= \frac{1}{\tau}\int_{t_{j-1}}^{t_j} k(t_n - s) \diff s  .
\end{equation*}
The idea behind is the approximation
\begin{equation*}
\int_0^t k(t-s) v(s)\diff s \approx \sum_{j=1}^n \int_{t_{j-1}}^{t_j} k(t_n -s) \diff s \, v(t_j) \quad \text{for } t \in (t_{n-1},t_n] \ (n=1, 2, \dots , N) .
\end{equation*}

As we deal with kernels $k$ of exponential type given by (\ref{Volterra}), we can explicitly calculate
\begin{equation*}
 \gamma_i = \frac{e^{\lambda \tau}-1}{\tau}\,e^{-\lambda t_i} , \quad i = 1, 2, \dots , N ,
\end{equation*}
which immediately leads to the properties
\begin{equation}\label{Gamma1}
0 < \gamma_1=\frac{1-e^{-\lambda \tau}}{\tau} \to \lambda = k(0) \quad\text{as } \tau \to 0
\end{equation}
and
\begin{equation}\label{Gamma2}
 \gamma_{i+1}-\gamma_i= - \bigl(e^{\lambda \tau} - 1\bigr)\, \gamma_{i+1} , \quad i= 1, 2, \dots , N-1 .
\end{equation}

\subsection{Existence, uniqueness, and a priori estimates for the time discrete problem}

\noindent
In what follows, we show existence of solutions for the time discrete equation (\ref{Ktau}) and derive suitable a priori estimates.

\begin{theorem}[Time discrete problem] \label{aprioriestimate}
Let Assumptions~$(\mathbf{A})$ and $(\mathbf{B})$  be
fulfilled and let $u^0 \in V_B$, $v^0 \in H$,  $\{f^n\}_{n=1}^N =\{f_0^n\}_{n=1}^N + \{f_1^n\}_{n=1}^N \subset V_A'+H$ be
given. Then there is a unique solution $\{v^n\}_{n=1}^N\subset V = V_A \cap V_B$ to (\ref{scheme}). Moreover, there holds for $n=1, 2, \dots , N$
\begin{equation}\label{apriori}
  \begin{aligned}
    \|v^n\|_H^2 + \sum_{j=1}^n \|v^j-v^{j-1}\|_H^2 + \mu_A \tau\sum_{j=1}^n \|v^j\|_{V_A}^p
+ \frac{T}{e^{\lambda T} -1}\, \left\| \left(K^\tau_{u^0} v\right)^n\right\|_{V_B}^2 \\
    \le c\,\left(1 + \|u^0\|_{V_B}^2 +\|v^0\|_H^2
+ \tau \sum_{j=1}^N \|f_0^j\|_{V_A'}^{p'}+ \left( \tau \sum_{j=1}^N \|f_1^j\|_{H} \right)^2 \right).
  \end{aligned}
\end{equation}
\end{theorem}

\begin{proof}
For better readability, we write $K^\tau$ instead of $K^\tau_{u^0}$ during this proof. We commence with proving existence and uniqueness of a solution step by step.
In the $n$th step, (\ref{scheme1}) is equivalent to determine $v^n$ from
$\{v^j\}_{j=1}^{n-1}\subset V$ and the data of the problem by solving
\begin{equation*}
\left(\frac{1}{\tau} I + A + \tau \gamma_1 B\right) v^n = f^n + \frac{1}{\tau}\, v^{n-1} - Bu^0 - \tau \sum_{j=1}^{n-1} \gamma_{n-j+1}Bv^{j} .
\end{equation*}
Because of the continuous embeddings (\ref{embed}), we know that the right-hand side of the foregoing relation is in $V'$. The operator
$M := \frac{1}{\tau} I + A + \tau \gamma_1 B$ is easily shown to be hemicontinuous, coercive, and strictly monotone as a mapping of $V$ into $V'$.
Here, we make use of the fact that $\gamma_1 > 0$. In particular, we observe that for all $w\in V$
\begin{equation*}
  \langle Mw , w \rangle \ge \frac{1}{\tau}\, \|w\|_H^2 + \mu_A \, \|w\|_{V_A}^p - c_A + \tau\gamma_1\mu_B \, \|w\|_{V_B}^2
\ge c \, \|w\|_{V}^2 - c .
\end{equation*}
The famous theorem of Browder and Minty (see, e.g.,
Zeidler~\cite[Thm.~26.A on p.~557]{ZeidlerIIB}) now provides existence of a solution $v^n \in V$; uniqueness immediately follows
from the strict monotonicity.

For proving the a priori estimate, we test (\ref{scheme1}) by $v^n\in V$ ($n=1, 2, \dots , N$). With
\begin{equation}\label{aba}
(a-b) a  = \frac{1}{2} \left( a^2 - b^2 + (a-b)^2 \right) , \quad a, b \in \mathbb{R} ,
\end{equation}
the coercivity of $A$, and Young's inequality, we find for $n=1, 2,
\dots , N$
\begin{equation*}
  \begin{split}
    &\frac{1}{2\tau}\left(\|v^n\|_H^2 - \|v^{n-1}\|_H^2 + \|v^n - v^{n-1}\|_H^2\right) +
\mu_A\, \|v^n\|_{V_A}^p  - c_A + \langle (B K^\tau v)^n, v^n \rangle
\\
&\le \langle f^n , v^n \rangle \le \|f_0^n\|_{V_A'} \|v^n\|_{V_A} +\|f_1^n\|_{H} \|v^n\|_{H}
\le \frac{\mu_A}{2}\, \|v^n\|_{V_A}^p +
c\, \|f_0^n\|_{V_A'}^{p'} +  \|f_1^n\|_{H}\, \|v^n\|_{H} .
  \end{split}
\end{equation*}
By summing up, we obtain for $n=1, 2, \dots , N$
\begin{equation*}
  \begin{split}
    &\|v^n\|_H^2 + \sum_{j=1}^n \|v^j - v^{j-1}\|_H^2 + \mu_A\tau \sum_{j=1}^n \|v^n\|_{V_A}^p
+ 2\tau \sum_{j=1}^n  \langle \left(B K^\tau v\right)^j, v^j \rangle\\
    &\leq c \left(1+\|v^0\|_{H}^2 + \tau \sum_{j=1}^n \|f_0^j\|_{V_A'}^{p'} \right) + 2\tau \sum_{j=1}^n \|f^j_1\|_H \, \|v^j\|_H ,
  \end{split}
\end{equation*}
where
\begin{equation*}
2\tau \sum_{j=1}^n \|f^j_1\|_H \, \|v^j\|_H \le
2\tau \sum_{j=1}^n \|f^j_1\|_H \, \max_{j=1, \dots , n} \|v^j\|_H
\le
2\left( \tau \sum_{j=1}^n \|f^j_1\|_H   \right)^2 + \frac{1}{2}\,
\max_{j=1, \dots , n} \|v^j\|_H^2 .
\end{equation*}

We observe that (\ref{Ktau}) together with (\ref{Gamma1}), (\ref{Gamma2}) implies for $n=2, 3, \dots , N$
\begin{equation}\label{KFunkDisk}
\frac{1}{\tau}\left( (K^\tau v)^n - (K^\tau v)^{n-1}  \right)
= \gamma_1 v^n + \sum_{j=1}^{n-1} \left(\gamma_{n-j+1} - \gamma_{n-j}\right) v^j
= \frac{e^{\lambda \tau}-1}{\tau} \Big(
v^n - \bigr((K^\tau v)^n -  u^0\bigr) \Big)\;\! ,
\end{equation}
which is the discrete analogue of the crucial relation (\ref{GleichungK}).
Note that $(e^{\lambda \tau}-1)/\tau \to \lambda$ as $\tau \to 0$.
With $(K^\tau v)^0 := u^0$, the relation (\ref{KFunkDisk}) remains true for $n=1$.

Resolving (\ref{KFunkDisk}) for $v^n$ gives, together with (\ref{aba}) (recall that $B:V_B \to V_B'$ induces an inner product on $V_B$),
\begin{align*}
2\tau\sum_{j=1}^n  \langle \left(B K^\tau v\right)^j, v^j  \rangle
&= \frac{2\tau}{e^{\lambda \tau} -1}\sum_{j=1}^n
\left\langle (B K^\tau v)^j ,( K^\tau v)^j -( K^\tau v)^{j-1} \right\rangle
+ 2\tau\sum_{j=1}^n \langle \left(B K^\tau v\right)^j, \left(K^\tau v\right)^j -u^0  \rangle
\\
&\ge
\frac{\tau}{e^{\lambda \tau} -1}\Big(\|\left(K^\tau v\right)^n\|_B^2 -
\|\left(K^\tau v\right)^0\|_B^2\Big) + \tau \sum_{j=1}^n \|\left(K^\tau v\right)^j\|_B^2
- T\, \|u^0\|_B^2 ,
\end{align*}
where
\begin{equation*}
\frac{\tau}{e^{\lambda \tau} -1} \ge  \frac{T}{e^{\lambda T} -1} .
\end{equation*}

Altogether, we find for $n=1, 2, \dots , N$
\begin{equation} \label{aprioriproof}
  \begin{split}
    &\|v^n\|_H^2 + \sum_{j=1}^n \|v^j - v^{j-1}\|_H^2 + \mu_A\tau \sum_{j=1}^n \|v^n\|_{V_A}^p
+  \frac{T}{e^{\lambda T} -1} \, \|\left(K^\tau v\right)^n\|_B^2
\\
    &\le c \left(1 + \|u^0\|_{V_B}^2 +\|v^0\|_{H}^2 + \tau \sum_{j=1}^N \|f_0^j\|_{V_A'}^{p'} + \left( \tau \sum_{j=1}^N \|f^j_1\|_H   \right)^2\right) + \frac{1}{2}\,
\max_{j=1, \dots , n} \|v^j\|_H^2  .
  \end{split}
\end{equation}
Taking the maximum over all $n\in\{1,...,N\}$ first on the right-hand side and then on the left-hand side leads to
\begin{equation*}
  \max_{j=1,...,N} \|v^j\|^2_H \leq C+ \frac{1}{2}\,
\max_{j=1, \dots , N} \|v^j\|_H^2 ,
\end{equation*}
where $C>0$ depends on $u^0$, $v^0$, and $\{f^n\}$. This yields $$\max_{j=1,...,N} \|v^j\|^2_H \leq 2C,$$ which, inserted in \eqref{aprioriproof}, proves the assertion.
\qed
\end{proof}

\subsection{Convergence of approximate solutions and existence of a generalised solution}

\noindent
Let $\{N_\ell\}$ be a sequence of positive integers such that $N_\ell \to \infty$ as $\ell \to \infty$. We consider the corresponding sequence of time discrete problems (\ref{scheme}) with step sizes $\tau_\ell = T/N_\ell$, starting values $v^0_\ell \in H$
with $v^0_\ell \to v_0$ in $H$ as well as $u^0_\ell \in V_B$
with $u^0_\ell \to u_0$ in $V_B$, and right-hand sides $\{f^n\}_{n=1}^{N_\ell} \subset V_A'+H$ given by
\begin{equation}\label{discrete2}
f^n=f_0^n+f_1^n,\quad f_{0,1}^n := \frac{1}{\tau_\ell} \int_{t_{n-1}}^{t_n} f_{0,1}(t) \diff t ,
\end{equation}
which is well-defined for $f=f_0+f_1\in  L^{p'}(0,T;V_A')+L^1(0,T;H)$.
As a slight abuse of notation, in general, we do not call the dependence of $u^n, v^n, f^n$ and of the time instances $t_n$ on $\ell$.

Let $\{v^n\}_{n=1}^{N_\ell} \subset V$ denote the solution to
(\ref{scheme}) with step size $\tau_\ell$. We  then consider the piecewise constant functions $v_\ell$ with $v_\ell(t) = v^n$ for $t \in
(t_{n-1},t_n]$ ($n=1, 2, \dots , N_\ell$) and $v_\ell(0)=v^1$. 
Moreover, let $\hat{v}_\ell$ be the piecewise affine-linear interpolation
of the points $(t_n,v^n)$ ($n=0, 1, \dots , N_\ell$) and $f_\ell$ be the piecewise constant function with $f_\ell(t)=f^n$ for $t \in (t_{n-1},t_n]$ ($n=1, 2, \dots , N_\ell$) and $f_\ell(0)=f^1$.

Regarding the integral operator $K$, we define for any integrable function $w$ the piecewise constant function $K_\ell w$ by means of
\begin{equation*}
 (K_\ell w)(t)= u^0_\ell + \int_0^{t_n} k(t_n-s) w(s)\diff s \quad
 \text{if } t\in (t_{n-1}, t_n] \; (n=1,\dots, N_\ell) ,
\end{equation*}
with $(K_\ell w)(0) := u^0_\ell$. As an immediate consequence, we obtain with \eqref{Ktau} and \eqref{DefGamma}
\begin{equation}\label{K_ell}
(K_\ell v_\ell) (t)= \left(K^{\tau_\ell}_{u^0_\ell} v\right)^n \quad\text{if }t\in (t_{n-1}, t_n]\; (n=1,\dots, N_\ell) .
\end{equation}

We are now able to state the main result.

\begin{theorem}[Existence via convergence]\label{MainThm}
Let Assumptions {\bf$(\mathbf{A})$} and {\bf $(\mathbf{B})$} be fulfilled, let $u_0\in V_B$, $v_0\in H$, and $f\in L^{p'}(0,T;V_A') + L^1(0,T;H)$.
Then there exists a solution $v \in L^p(0,T;V_A) \cap \mathscr{C}_w([0,T];H)$ to problem (\ref{abstract1}) with $Kv\in \mathscr{C}_w([0,T];V_B)$ such that (\ref{abstract1a}) holds in the sense of  $L^{p'}(0,T;V_A') + L^1(0,T;H)$.

Passing to a subsequence if necessary,
both the piecewise constant prolongation $v_{\ell}$ and the piecewise affine-linear prolongation $\hat v_{\ell}$ of the approximate solutions to
(\ref{scheme})
converge weakly* in $L^\infty(0,T;H)$ to $v$ as $\ell\to \infty$. Furthermore, again passing to a subsequence if necessary and as $\ell\to\infty$, $v_\ell$ converges weakly  to $v$ in $L^p(0,T;V_A)$ and the approximation $K_\ell v_{\ell}$ of the memory term converges weakly* to $Kv$ in $L^\infty(0,T;V_B)$. The approximation $\hat v_{\ell}'$ of the time derivative converges to $v'\in L^{p'}(0,T;V_A') + L^1(0,T;H)+L^\infty(0,T;V_B')$ in the sense that $\langle \hat v_{\ell}' ,w\rangle \to \langle v',w\rangle$ for all $w\in L^p(0,T;V_A)\cap L^\infty(0,T;H)\cap L^1(0,T;V_B)$.\footnote{Here, we slightly abuse the notation of the duality pairing, since $L^{p'}(0,T;V_A') + L^1(0,T;H)+L^\infty(0,T;V_B')$ is not the dual space of $L^p(0,T;V_A)\cap L^\infty(0,T;H)\cap L^1(0,T;V_B)$. However, we can consider the sum of the duality pairings between $L^{p'}(0,T;V_A') + L^1(0,T;H)$ and $L^1(0,T;V_B)$ and their respective dual spaces.}
\end{theorem}

Note that if $\left\{\tau_\ell \|v^0_\ell\|_{V_A}^p\right\}$ is bounded then also
$\hat v_\ell$ converges weakly in $L^p(0,T;V_A)$ to $v$ as $\ell\to\infty$ (passing to a subsequence if necessary).

The proof of Theorem~\ref{MainThm} will be prepared by the following integration-by-parts formula.

\begin{lemma}
\label{PartInt}
Let $u_0\in V_B$ and $w\in L^p(0,T; V_A)\cap L^\infty(0,T;H)$ such that $Kw\in L^\infty(0,T;V_B)$ and $w'+BKw\in L^{p'}(0,T; V_A')+L^1(0,T;H)$. Then for all $t \in [0,T]$, there holds
\begin{equation}\label{PartInt2}
\begin{split}
&\int^t_0 \langle w'(s) + (BKw)(s), w(s)\rangle \diff s = \\&\frac{1}{2}\|w(t)\|^2_H - \frac{1}{2}\|w(0)\|^2_H + \frac{1}{2\lambda}\|(Kw)(t)\|^2_{B} -\frac{1}{2\lambda}\|u_0\|^2_{B}  - \int_0^t \langle (BKw)(s), u_0\rangle \diff s + \int_0^t \|(Kw)(s)\|^2_{B}\diff s.
\end{split}
\end{equation}
\end{lemma}
\begin{proof}
The difficulty in proving (\ref{PartInt2}) is that neither $w'\in L^{p'}(0,T;V_A')+L^1(0,T;H)$ nor $BKw\in L^{p'}(0,T; V_A')+L^1(0,T;H)$ can be assumed but only the sum $w'+BKw$ is known to be in $L^{p'}(0,T;V_A')+L^1(0,T;H)$. Hence it is not possible to split the sum on the left-hand side and to carry out integration by parts separately for both terms.

As $w'+BKw\in L^{p'}(0,T;V_A')+L^1(0,T;H)$ and $Kw\in L^\infty(0,T;V_B')$, it is easy to show that $w' \in L^{p'}(0,T;V_A')+L^1(0,T;V_B') \subset L^1(0,T;V_A'+V_B')$ and thus $w\in \mathscr{AC}([0,T];V_A'+V_B')$. In view of Lions and Magenes~\cite[Ch.~3, Lemma~8.1]{LionsMagenes}, we find $w\in \mathscr{C}_w([0,T];H)$. Because of (\ref{GleichungK}), we also see that $(Kw)' \in L^1(0,T; V_A+V_B)$ and thus $Kw \in \mathscr{AC}([0,T];V_A+V_B)$, which implies $Kw\in \mathscr{C}_w([0,T];V_B)$. Because of \eqref{GleichungK}, we also have $(Kw)'\in\mathscr{C}_w([0,T];H)$. This justifies evaluating $\Vert w(t)\Vert_H$, $\Vert Kw(t)\Vert_B$, and $\Vert (Kw)'(t)\Vert_H$ for fixed $t$.

Further, we remark that $Kw-u_0\in L^\infty(0,T;V)$ since $w\in L^p(0,T; V_A)\cap L^\infty(0,T;H)$. Due to \eqref{GleichungK}, we thus have $(Kw)'\in L^p(0,T;V_A)\cap L^\infty(0,T;H)$. Differentiating \eqref{GleichungK} yields
\begin{equation*}
  (Kw)''(t)=\lambda \left( w'(t)-(Kw)'(t) \right)
\end{equation*}
for almost all $t\in (0,T)$. Therefore, we have
\begin{equation*}
  \begin{split}
   \langle w'+BKw,w\rangle = \left\langle \frac{1}{\lambda} (Kw)''+(Kw)'+BKw,\frac{1}{\lambda} (Kw)'+Kw-u_0\right\rangle.
  \end{split}
\end{equation*}
Since $\frac{1}{\lambda} (Kw)''+(Kw)'+BKw\in L^{p'}(0,T; V_A')+L^1(0,T;H)$ and $Kw-u_0\in L^\infty(0,T;V)$, we are allowed to split the right-hand side into
\begin{equation} \label{3parts}
  \begin{split}
    &\left\langle \frac{1}{\lambda} (Kw)''+(Kw)'+BKw,\frac{1}{\lambda} (Kw)'\right\rangle +\left\langle \frac{1}{\lambda} (Kw)''+(Kw)'+BKw,Kw-u_0\right\rangle\\
    &= \left\langle \frac{1}{\lambda} (Kw)''+BKw,\frac{1}{\lambda} (Kw)'\right\rangle +\frac{1}{\lambda} \Vert (Kw)'\Vert^2_{L^2(0,T;H)} +\left\langle \frac{1}{\lambda} (Kw)''+(Kw)'+BKw,Kw-u_0\right\rangle.\\
  \end{split}
\end{equation}
For the first term on the right-hand side, we apply the result of Strauss \cite[Theorem 4.1]{Strauss66} providing the integration-by-parts formula
\begin{equation*}
  \begin{split}
    \left\langle \frac{1}{\lambda} (Kw)''+BKw,\frac{1}{\lambda} (Kw)'\right\rangle = \frac{1}{2\lambda^2} \left( \Vert (Kw)'(T)\Vert^2_H- \Vert (Kw)'(0)\Vert^2_H + \lambda \Vert (Kw)(T)\Vert^2_B - \lambda \Vert u_0\Vert^2_B \right).
  \end{split}
\end{equation*}
Since $(Kw)'\in L^p(0,T;V_A)\cap L^\infty(0,T;H)$, $BKw\in L^\infty(0,T;V_B')$, and $Kw-u_0\in L^\infty(0,T;V)$, the third term on the right-hand side of \eqref{3parts} yields
\begin{equation*}
  \begin{split}
    &\left\langle \frac{1}{\lambda} (Kw)''+(Kw)'+BKw,Kw-u_0\right\rangle \\
    &= \left\langle \frac{1}{\lambda} (Kw)'',Kw-u_0\right\rangle+\left\langle(Kw)',Kw-u_0\right\rangle+\left\langle BKw,Kw-u_0\right\rangle \\
    &= \frac{1}{\lambda}\left\langle (Kw-u_0)'(T), (Kw-u_0)(T)\right\rangle - \frac{1}{\lambda}\left\langle (Kw-u_0)'(0), (Kw-u_0)(0)\right\rangle -\frac{1}{\lambda} \Vert (Kw-u_0)'\Vert^2_{L^2(0,T;H)} \\
    &\hspace{1.5cm} +\frac{1}{2} \Vert (Kw)(T)-u_0\Vert^2_H -\frac{1}{2} \Vert (Kw)(0)-u_0\Vert^2_H +\Vert Kw\Vert^2_{L^2(0,T;B)} - \left\langle BKw, u_0\right\rangle.
  \end{split}
\end{equation*}
Altogether, we have
\begin{equation}\label{partInt}
  \begin{split}
    \langle w'+BKw,w\rangle &= \frac{1}{2\lambda^2} \left( \Vert (Kw)'(T)\Vert^2_H- \Vert (Kw)'(0)\Vert^2_H + \lambda \Vert (Kw)(T)\Vert^2_B - \lambda \Vert u_0\Vert^2_B \right) \\
    &\hspace{1.5cm} +\frac{1}{\lambda}\left\langle (Kw-u_0)'(T), (Kw-u_0)(T)\right\rangle +\frac{1}{2} \Vert (Kw)(T)-u_0\Vert^2_H \\
    &\hspace{6cm}+\Vert Kw\Vert^2_{L^2(0,T;B)} - \left\langle BKw, u_0\right\rangle.
  \end{split}
\end{equation}
Due to \eqref{GleichungK}, there holds
\begin{equation*}
  \begin{split}
    \frac{1}{2} \Vert w(T)\Vert^2_H = \frac{1}{2\lambda^2} \Vert (Kw)'(T)\Vert^2_H + \frac{1}{2} \Vert (Kw)(T)-u_0\Vert^2_H + \frac{1}{\lambda}\left\langle (Kw)'(T), (Kw)(T)-u_0\right\rangle
  \end{split}
\end{equation*}
as well as
\begin{equation*}
  \frac{1}{2} \Vert w(0)\Vert^2_H = \frac{1}{2\lambda^2} \Vert(Kw)'(0)\Vert^2_H.
\end{equation*}
Inserting this into \eqref{partInt} yields the desired integration-by-parts formula for $t=T$. However, everything above remains true for an arbitrary $t\in [0,T]$, which proves the assertion.\qed
\end{proof}

\begin{proof}[of Theorem~\ref{MainThm}]
For simplicity, we write $\tau$ and $N$ instead of $\tau_\ell$ and $N_\ell$, respectively, during this proof. Since, by assumption, $v_\ell^0\to v_0$ in $H$, $u^0_\ell\to u_0$ in $V_B$, and since $f_\ell\to f$ in $L^{p'}(0,T;V_A') + L^1(0,T;H)$,
the right-hand side of the a priori estimate (\ref{apriori}) stays bounded as $\ell \to \infty$.
Since
\begin{gather*}
    \| v_\ell \|_{L^\infty(0,T;H)} = \max_{n= 1, 2, \dots , N } \|v^n\|_H , \quad \quad
    \|\hat{v}_\ell\|_{L^\infty(0,T;H)} = \max_{n= 0, 1, \dots , N } \|v^n\|_H , \quad \\
    \|v_\ell\|_{L^p(0,T;V_A)} = \left( \tau \sum_{n=1}^{N} \|v^n\|_{V_A}^p\right)^{1/p} ,
\end{gather*}
the a priori estimate (\ref{apriori}) implies
that the sequence $\{v_\ell\}$ is bounded in $L^\infty(0,T;H)\cap L^p(0,T;V_A)$,  the sequence $\{\hat v_\ell\}$ is bounded in $L^\infty(0,T;H)$, and -- in view of (\ref{K_ell}) -- the sequence
$\{K_\ell v_\ell\}$ is bounded in $L^\infty(0,T;V_B)$.
Due to the growth condition on $A$, we also conclude that the sequence
$\{A v_\ell\}$ is bounded in $L^{p'}(0,T;V_A')$.

Since $L^\infty(0,T;H)$ and $L^\infty(0,T;V_B)$ are the duals of the separable normed spaces $L^1(0,T;H)$ and $L^1(0,T;V_B')$, respectively, and since $L^p(0,T;V_A)$ as well as $L^{p'}(0,T;V_A')$ are reflexive Banach spaces, there exists a common subsequence, again denoted by $\ell$, and elements
$v\in L^\infty(0,T;H)\cap L^p(0,T;V_A)$, $\hat v \in L^\infty(0,T;H)$,
$u\in L^\infty(0,T;V_B)$, $a\in L^{p'}(0,T;V_A')$ such that,
as $\ell\to \infty$,
\begin{align*}
& v_{\ell} \wsconv v \quad\textnormal{in } L^\infty (0,T;H) , \quad
v_{\ell} \wconv v \quad\textnormal{in } L^p (0,T;V_A) , \quad
\hat v_{\ell} \wsconv \hat v \quad\textnormal{in } L^\infty (0,T;H) ,\\
& K_{\ell} v_{\ell} \wsconv u \quad\textnormal{in } L^\infty (0,T;V_B) , \quad Av_{\ell} \wconv a \quad\textnormal{in } L^{p'} (0,T;V_A') ,
\end{align*}
see, e.g., Br\'ezis~\cite[Coroll.~III.26, Thm. III.27]{Brezis}.
Since $B:V_B \to  V_B'$ is linear and bounded,  we also find that $BK_{\ell} v_{\ell} \wsconv Bu$ in $L^\infty(0,T;V_B')$.

A short calculation shows that
\begin{equation*}
\|v_{\ell} - \hat v_{\ell}\|_{L^2(0,T;H)}^2 = \frac{\tau}{3}\, \sum_{j=1}^{N} \|v^j - v^{j-1}\|_H^2 ,
\end{equation*}
which tends to $0$ as ${\ell}\to \infty$ because of the a priori estimate (\ref{apriori}). Hence $v=\hat v$.

Let us show that $u=Kv$ and thereby $Kv\in L^\infty(0,T;V_B)$.
We already know that $K_{\ell} v_{\ell} \wsconv u$ in $L^\infty(0,T;V_B)$ and thus $K_{\ell} v_{\ell} \wconv u$ in $L^2(0,T;H)$.
The integral operator $\hat{K}$ defined by $\hat{K} w := Kw - u_0$
can easily be shown to be a linear, bounded and thus
weakly-weakly continuous mapping of $L^2 (0,T;H)$ into itself (see also Lemma~\ref{LemmaK}). Therefore, in view of $v_{\ell} \wsconv v$ in $L^\infty(0,T;H)$ and thus $v_{\ell} \wconv v$ in $L^2(0,T;H)$, we
find that $Kv_{\ell} - Kv = \hat{K} v_{\ell} - \hat{K} v  \wconv 0$ in $L^2(0,T;H)$.
Because of
\begin{equation*}
 u - Kv = u - K_{\ell} v_{\ell} + K_{\ell} v_{\ell} - Kv_{\ell}
 + Kv_{\ell} - Kv ,
\end{equation*}
it thus remains to prove that $K_{\ell} v_{\ell} - Kv_{\ell}$ converges at least weakly in $L^2(0,T;H)$ to $0$ in order to prove $u=Kv$.
For $w\in L^2(0,T;H)$ and $t\in (t_{n-1}, t_n]$ ($n = 1, 2, \dots, N$), we have
\begin{equation*}
(K_{\ell} w) (t) - (Kw) (t)= u^0_\ell-u_0 +\int_0 ^{t_n}\Big(k(t_n-s) - k(t-s)\Big)\;\! w(s) \diff s +\int_t^{t_n} k(t-s)w(s)\diff s .
\end{equation*}
Since $|k(t-s)| = \lambda e^{-\lambda(t-s)} \leq \lambda e^{\lambda\tau}$ for $s\in(t,t_n)\subset (t_{n-1},t_n)$ as well as
\begin{equation*}
\vert k(t_n-s) - k(t-s)\vert \leq \lambda\;\! \big(e^{\lambda \tau} -1\big) \leq c \;\! \tau \quad\text{for } s\in (0,t_n),
\end{equation*}
the Cauchy--Schwarz inequality yields
\begin{equation*}
\|K_{\ell} w - K w\|_{L^2(0,T;H)} \le \sqrt{T}\Vert u^0_\ell-u_0\Vert_H+ c \tau \, \|w\|_{L^2(0,T;H)} .
\end{equation*}
Since $u^0_\ell\to u_0$ in $V_B$ and $v_{\ell}$ is bounded in $L^2(0,T;H)$, we immediately get
$K_{\ell} v_{\ell} - Kv_{\ell} \to 0$ in $L^2(0,T;H)$ as $\ell\to\infty$ (and thus $\tau\to0$), which is the last step to prove $u=Kv \in L^\infty(0,T;V_B)$.

With the help of the piecewise constant and piecewise linear interpolation of the corresponding grid functions, we can now rewrite the numerical scheme (\ref{scheme1}) as
\begin{equation}\label{P_disc2}
 \hat v_{\ell}' + Av_{\ell} + BK_{\ell} v_{\ell}=f_{\ell} .
\end{equation}
For any $w\in V$ and $\phi \in \mathscr{C}_c^\infty (0,T)$, we thus obtain
\begin{equation*}
-\int^T_0 \langle \hat v_{\ell} (t), w\rangle\;\! \phi'(t)\diff t = \int^T_0 \Big\langle f_{\ell}(t) - Av_{\ell} (t)-(BK_{\ell} v_{\ell}) (t) , w\Big\rangle\;\! \phi(t) \diff t .
\end{equation*}
Since $f_{\ell} \to f$ in $L^{p'}(0,T;V_A') + L^1(0,T;H)$, and employing the (weak and weak*) convergence just shown, we can pass to the limit in the foregoing equation to come up with
\begin{equation*}
-\int^T_0 \langle v(t), w\rangle\;\! \phi'(t)\diff t = \int^T_0 \Big\langle f(t) - a(t)-(BK v)(t) , w\Big\rangle\;\! \phi(t)\diff t .
\end{equation*}
This shows that $v$ is differentiable in the weak sense with
\begin{equation}\label{P2}
v'=f-a-BKv \in L^{p'}(0,T;V_A')+ L^\infty(0,T;V_B')+ L^1(0,T;H) .
\end{equation}

Let us summarise the regularity properties of the limits we have proven so far: As $v\in  L^p(0,T;V_A) \cap L^\infty(0,T;H)$ and $v'\in L^{p'}(0,T;V_A')+ L^\infty(0,T;V_B')+L^1(0,T;H) \subset L^1(0,T;V')$, we conclude that
$v\in \mathscr{AC}([0,T];V')$. This implies $v\in \mathscr{C}_w([0,T]; H)$ (see Lions and Magenes~\cite[Ch.~3, Lemma~8.1]{LionsMagenes}).
Likewise, from $Kv\in L^\infty(0,T;V_B)$ and $(Kv)'=\lambda v-\lambda Kv + \lambda u_0\in L^{p}(0,T;V_A) + L^\infty(0,T;V_B) \subset L^1(0,T;V_A+V_B)$ (see (\ref{GleichungK})), we deduce that $Kv\in \mathscr{AC}([0,T]; V_A+V_B)$ and thus  $Kv\in \mathscr{C}_w([0,T]; V_B)$. We remark that, in addition, one can easily show that $Kv-u_0\in W^{1,p}(0,T;V_A)$.

By construction, we have $\hat v_{\ell} (0)=v_{\ell}^0 \to v_0$ in $H$ and $\hat v_{\ell} (T) = v^{N}$. In view of the a priori estimate (\ref{apriori}), the sequence $\{\hat v_{\ell} (T)\}$ is bounded in $H$ and hence possesses a subsequence, again denoted by $\ell$, that converges weakly in $H$ to an element $v_T\in H$. We shall show that $v_T = v(T)$.

For any $\phi\in \mathscr{C}^1([0,T])$, $w\in V$, we test (\ref{P2}) with the function $t\mapsto \phi(t)w \in \mathscr{C}^1([0,T];V)$, which allows us to integrate by parts. Together with (\ref{P_disc2}), we find
\begin{equation} \label{endpoint_v}
  \begin{aligned}
    &  \Big(v(T),w\Big)\;\!\phi(T) - \Big(v(0), w\Big)\;\!\phi(0) = \int_0^T \Big\langle v'(t), w\Big\rangle\;\! \phi(t) \diff t + \int_0^T\Big \langle v(t), w\Big\rangle \;\!\phi'(t)\diff t\\
    &=\int_0^T \Big\langle f(t) -f_{\ell}(t) + Av_{\ell}(t) - a(t) + (BK_{\ell}v_{\ell})(t) - (BKv)(t), w\Big\rangle\;\! \phi(t) \diff t\\
    &\phantom{MMMMMM}+ \int_0^T \langle v(t), w\rangle\;\! \phi'(t) \diff t+ \int_0^T \langle \hat v_{\ell}'(t), w\rangle\;\! \phi(t) \diff t \\
    &= \int_0^T \Big\langle f(t) -f_{\ell}(t) + Av_{\ell}(t) - a(t) + (BK_{\ell}v_{\ell})(t) - (BKv)(t), w\Big\rangle\;\! \phi(t) \diff t\\
    &\phantom{MMMMMM}+ \int_0^T \langle v(t) - \hat v_{\ell}(t), w\rangle\;\! \phi'(t) \diff t+ \Big(\hat v_{\ell}(T), w\Big)\;\! \phi(T) - \Big(\hat v_{\ell}(0), w\Big)\;\! \phi(0) .
  \end{aligned}
\end{equation}
The first two terms on the right-hand side tend to zero as $\ell\to \infty$. This yields
\begin{equation*}
\Big(v(T),w\Big)\;\!\phi(T) - \Big(v(0), w\Big)\;\!\phi(0)=\Big( v_T, w\Big)\;\! \phi(T) - \Big( v_0, w\Big) \;\!\phi(0) \quad \textnormal{for all }w\in V .
\end{equation*}
Choosing $\phi$ such that $\phi(T)=0$ and $ \phi(0)=0$, respectively, shows that $v(0)= v_0$ and  $v(T)=v_T$, respectively.

By definition, we have $(K_{\ell} v_{\ell}) (0) =u^0_\ell \to u_0$ in $V_B$ and $(K_{\ell} v_{\ell} )(T) = \left(K^\tau\!\textsubscript{\hspace{-0.05cm}$u^0_\ell$}\;\! v\right)^{N}$. Again, the a priori estimate yields the boundedness of $\{(K_{\ell} v_{\ell}) (T)\}$ in $V_B$ and thus the existence of a weakly convergent subsequence, again denoted by $\ell$, and an element $u_T\in V_B$ such that $ (K_{\ell} v_{\ell} )(T)\wconv u_T$ in $V_B$.

We shall show that $u_T=(Kv)(T)$. Analogously to \eqref{endpoint_v}, choosing $\phi(t)=\frac{t}{T}$, $t\in [0,T]$, we have for any $w\in V$
\begin{equation*}
  \Big( (Kv)(T), w \Big) = \int_0^T \Big\langle (Kv)'(t), w\Big\rangle \frac{t}{T} \diff t + \int_0^T\Big \langle (Kv)(t), w\Big\rangle \frac{1}{T}\diff t.
\end{equation*}
Using \eqref{GleichungK}, we obtain
\begin{equation} \label{endpoint_Kv}
  \begin{split}
    \Big( (Kv)(T), w \Big) &= \lambda \int_0^T \Big\langle v(t)-((Kv)(t)- u_0), w\Big\rangle \frac{t}{T} \diff t + \int_0^T\Big \langle (Kv)(t), w\Big\rangle \frac{1}{T}\diff t \\
    &=\lim_{\ell\to\infty} \left(\lambda \int_0^T \Big\langle v_{\ell}(t)-\big( (K_{\ell}v_{\ell})(t)- u^0_\ell\big), w\Big\rangle \frac{t}{T} \diff t + \int_0^T\Big \langle (K_{\ell}v_{\ell})(t), w\Big\rangle \frac{1}{T}\diff t \right) .
  \end{split}
\end{equation}
For the first integral, \eqref{KFunkDisk} yields
\begin{equation*}
  \begin{split}
    \int_0^T \Big\langle v_{\ell}(t) -\big( (K_{\ell}v_{\ell})(t)- u^0_\ell\big),w \Big\rangle \frac{t}{T} \diff t &= \sum_{j=1}^{N} \int_{t_{j-1}}^{t_j} \left\langle v^j-\left(\left(K^{\tau}_{u^0_\ell}v\right)^j - u^0_\ell\right),w\right\rangle\frac{t}{T} \diff t \\
    &= \sum_{j=1}^{N} \int_{t_{j-1}}^{t_j} \frac{1}{e^{\lambda\tau}-1} \left\langle \left(K^{\tau}_{u^0_\ell}v\right)^j -\left(K^\tau_{u^0_\ell} v\right)^{j-1},w\right\rangle\frac{t}{T} \diff t. \\
  \end{split}
\end{equation*}
For better readability, we write $K^\tau$ instead of $K^\tau\!\textsubscript{\hspace{-0.05cm}$u^0_\ell$}$ for the rest of this proof. Splitting up the sum and shifting the index in the second sum yields
\begin{equation*}
  \begin{split}
    &\sum_{j=1}^{N} \int_{t_{j-1}}^{t_j} \Big\langle \left(K^{\tau}v\right)^j -\left(K^\tau v\right)^{j-1},w\Big\rangle\frac{t}{T} \diff t \\
    &= \langle (K^\tau v)^{N},w\rangle \int_{t_{N-1}}^{t_{N}} \frac{t}{T} \diff t - \langle u^0_\ell,w\rangle \int_0^\tau \frac{t}{T} \diff t + \sum_{j=1}^{N-1} \langle (K^\tau v)^j,w\rangle \left(\int_{t_{j-1}}^{t_{j}} \frac{t}{T} \diff t - \int_{t_{j}}^{t_{j+1}} \frac{t}{T} \diff t  \right) \\
    &= \frac{(2T-\tau)\;\!\tau}{2T}\;\! \langle (K^\tau v)^{N},w\rangle - \frac{\tau^2}{2T}\;\! \langle u^0_\ell,w\rangle - \frac{\tau^2}{T} \sum_{j=1}^{N} \langle (K^\tau v)^j,w\rangle + \frac{\tau^2}{T} \;\! \langle (K^\tau v)^N,w\rangle \\
    &= \left( \frac{(2T-\tau)\;\!\tau}{2T} + \frac{\tau^2}{T} \right) \langle (K_{\ell}v_{\ell})(T),w\rangle - \frac{\tau^2}{2T}\;\! \langle u^0_\ell,w\rangle - \frac{\tau}{T} \int_0^T \langle (K_{\ell}v_{\ell})(t),w\rangle \diff t.
  \end{split}
\end{equation*}
Inserting this into \eqref{endpoint_Kv} and using that $\big\{u^0_\ell\big\}$ is bounded in $V_B$, $(K_{\ell}v_{\ell})(T)\rightharpoonup u_T$ in $V_B$, $K_{\ell}v_{\ell}\wsconv Kv$ in $L^\infty(0,T;V_B)$ as well as $\lambda\tau\;\! (e^{\lambda\tau}-1)^{-1}\to 1$ as $\ell\to\infty$ (and thus $\tau\to0$), we end up with
\begin{equation*}
  \begin{split}
    \Big( (Kv)(T), w \Big) &= \lim_{\ell\to\infty}\left( \frac{\lambda\tau}{e^{\lambda\tau}-1} \left( \left( \frac{2T-\tau}{2T} +\frac{\tau}{T} \right) \langle (K_{\ell}v_{\ell})(T),w\rangle - \frac{\tau}{2T}\;\! \langle u^0_\ell,w\rangle  \right. \right. \\
    &\hspace{4cm} \left.\left. - \frac{1}{T} \int_0^T \langle (K_{\ell}v_{\ell})(t),w\rangle \diff t \right) + \int_0^T\Big \langle (K_{\ell}v_{\ell})(t), w\Big\rangle \frac{1}{T}\diff t \right) \\
    &= \langle u_T,w\rangle,
  \end{split}
\end{equation*}
which proves $u_T=(Kv)(T)$.

It remains to show that $a$ equals $Av$. We recall that $v_{\ell} \in L^\infty(0,T;V)$. Hence for arbitrary $w\in L^p(0,T;V_A)$, the monotonicity of $A$ implies
\[\begin{split}
   \langle Av_{\ell} , v_{\ell}  \rangle &=   \langle A v_{\ell} - A w, v_{\ell} -w  \rangle +   \langle Aw , v_{\ell} -w \rangle  +   \langle Av_{\ell} , w \rangle \geq   \langle Aw, v_{\ell}-w \rangle  +   \langle Av_{\ell}, w \rangle .
\end{split}\]
The right-hand side converges to
\[  \langle Aw , v-w \rangle +  \langle a, w \rangle .\]
This already implies
\[\liminf_{{\ell}\to\infty} \;\! \langle Av_{\ell} , v_{\ell}   \rangle  \geq \langle Aw  , v -w  \rangle  + \langle a  , w  \rangle .\]
We will show below
\begin{equation}\label{Minty1}
 \limsup_{{\ell}\to\infty} \;\!  \langle Av_{\ell}  , v_{\ell}   \rangle  \leq  \langle a  , v  \rangle,
\end{equation}
hence altogether proving
\begin{equation}\label{Minty2}
 \langle Aw   - a , v -w  \rangle \leq 0.
\end{equation}
Taking $w=v\pm rz$ for any $z\in L^p(0,T;V_A)$, $r>0$, and passing to the limit $r\to 0$, Lebesgue's theorem on dominated convergence together with the hemicontinuity and the growth condition of $A\colon V_A\to V_A'$ yield
\[
 \langle Av   - a , z  \rangle  = 0 \quad \textnormal{for all }z\in L^p(0,T;V_A),
\]
which proves the assertion.

To finish the proof of existence, it remains to show (\ref{Minty1}).
We start with
\begin{equation} \label{Av_lv_l}
   \langle Av_{\ell} , v_{\ell}  \rangle   =    \langle f_{\ell} , v_{\ell}  \rangle   -    \langle \hat v_{\ell}' , v_{\ell}  \rangle   -    \langle BK_{\ell}v_{\ell} , v_{\ell}  \rangle  .
\end{equation}
Since $f_{\ell}\to f$ in $L^{p'}(0,T;V_A')+L^1(0,T;H)$, the first term on the right-hand side tends to $  \langle f , v \rangle  $ as $\ell\to\infty$.
Using (\ref{aba}), we estimate the second term as
\begin{equation*}
  \left\langle \hat v_{\ell}' , v_{\ell} \right\rangle   = \sum_{j=1}^{N} \Big(v^j-v^{j-1}, v^j\Big) \geq \frac{1}{2}\| v^{N}\|_H^2 - \frac{1}{2}\| v^0_l\|_H^2.
\end{equation*}
As $v_l^0\to v_0=v(0)$ and $v^{N}\wconv v_T=v(T)$ in $H$, the weak lower semicontinuity of the norm provides
\[\liminf_{{\ell}\to\infty}    \Big\langle \hat v_{\ell}' , v_{\ell} \Big\rangle    \geq \frac{1}{2}\| v(T)\|_H^2 - \frac{1}{2}\| v(0)\|_H^2.\]
Regarding the last term on the right-hand side of \eqref{Av_lv_l}, note that only $v_{\ell} \wconv v$ in $L^p(0,T; V_A)$ and $BK_{\ell} v_{\ell} \wconv BKv$ in $L^2(0,T; V_B')$. Therefore, a finer examination of that term is necessary. To start with, a short calculation (see (\ref{KFunkDisk})) shows that
\[v^n = \frac{e^{-\lambda \tau}}{\tau \gamma_1} ((K^{\tau} v)^n - (K^{\tau} v)^{n-1}) + \frac{1-e^{-\lambda \tau}}{\tau \gamma_1} (K^{\tau} v)^n - \frac{1-e^{-\lambda \tau}}{\tau \gamma_1} u^0_\ell.\]
We conclude with (\ref{aba}) that
\begin{equation*}
  \begin{split}
    &  \langle BK_{\ell}v_{\ell} , v_{\ell} \rangle   = \tau  \sum_{j=1}^{N} \Big\langle (BK^{\tau} v)^j, v^j\Big\rangle\\
    &=  \frac{e^{-\lambda \tau}}{\gamma_1} \sum_{j=1}^{N} \Big\langle (BK^{\tau} v)^j, (K^{\tau} v)^j - (K^{\tau} v)^{j-1}\Big\rangle  + \frac{1-e^{-\lambda \tau}}{\gamma_1} \sum_{j=1}^{N} \Big\langle (BK^{\tau} v)^j, (K^{\tau} v)^j \Big\rangle\\
    &\hspace{7.5cm} - \frac{1-e^{-\lambda \tau}}{\gamma_1} \sum_{j=1}^{N}\Big\langle (BK^{\tau} v)^j, u^0_\ell\Big\rangle\\
    &\geq \frac{e^{-\lambda \tau}}{2\gamma_1} \Big(\|(K^{\tau} v)^{N}\|_B^2 -\|(K^{\tau} v)^0\|_B^2\Big)  + \frac{1-e^{-\lambda \tau}}{\tau \gamma_1} \Vert K_{\ell} v_{\ell}\Vert_{L^2(0,T;B)}^2- \frac{1-e^{-\lambda \tau}}{\tau\gamma_1} \int_0^T \Big\langle (BK_{\ell} v_{\ell})(t) , u^0_\ell\Big\rangle \diff t.
  \end{split}
\end{equation*}
Since $(K^{\tau} v)^0= u^0_\ell\to u_0=(Kv)(0)$ in $V_B$ and $(K^{\tau} v)^{N}= (K_{\ell}v_{\ell})(T)\wconv u_T=(Kv)(T)$ in $V_B$ as well as $K_{\ell}v_{\ell}\wconv Kv$ in $ L^2(0,T;V_B)$, the weak lower semicontinuity of the norms involved shows, by passing to the limit and employing $\gamma_1\to \lambda$ (see \eqref{Gamma1}),
\begin{equation*}
  \begin{split}
    &\liminf_{{\ell}\to\infty}  \;\!  \langle BK_{\ell}v_{\ell}, v_{\ell} \rangle    \geq \frac{1}{2\lambda} \|(Kv)(T)\|_{B}^2 - \frac{1}{2\lambda} \|(Kv)(0)\|_{B}^2 +  \Vert Kv\Vert_{L^2(0,T;B)}^2- \int_0^T \Big\langle (BKv)(t), u_0\Big\rangle \diff t.
  \end{split}
\end{equation*}
Altogether, we find
\begin{equation*}
\begin{split}
  &\limsup_{{\ell}\to\infty} \;\!  \langle Av_{\ell} , v_{\ell} \rangle \leq  \langle f, v\rangle -\frac{1}{2}\| v(T)\|_H^2 + \frac{1}{2}\| v(0)\|_H^2 -\frac{1}{2\lambda} \|(Kv)(T)\|_{B}^2 + \frac{1}{2\lambda} \|(Kv)(0)\|_{B}^2 \\
  &\hspace{7cm} -  \Vert Kv\Vert_{L^2(0,T;B)}^2+ \int_0^T \Big\langle (BKv)(t), u_0\Big\rangle \diff t.
\end{split}
\end{equation*}
We replace $f$ by $v'+a+BKv$ and remind that $v\in \mathscr{C}_w([0,T]; H)$ as well as $Kv\in \mathscr{C}_w([0,T]; V_B)$. Thus we are allowed to apply Lemma \ref{PartInt}, which yields (\ref{Minty1}).
%
\qed
\end{proof}

\section{Uniqueness and continuous dependence on the problem data}

\noindent
Using again the integration-by-parts formula of Lemma \ref{PartInt}, we are able to prove a stability and uniqueness result. Further, we are able to prove stability with respect to perturbations of the kernel parameter $\lambda$. In order to do so, we again have to prove an integration-by-parts formula using a different approach that requires less structure than the one used in the proof of Lemma~\ref{PartInt}.

\begin{theorem}[Stability]\label{ThmStability}
Let Assumptions {\bf$(\mathbf{A})$} and {\bf $(\mathbf{B})$} be fulfilled and let $u_0, \hat u_0\in V_B$, $v_0,\hat v_0\in H$, and $f,\hat{f}\in L^{p'}(0,T;V_A')+L^1(0,T;H)$. 
\begin{enumerate}[(i)]
\item Let $f-\hat f\in L^1(0,T;H)$ and let $v$, $\hat  v$ be solutions to (\ref{abstract1}) with data $({u}_0, {v}_0,  f)$ and $(\hat{u}_0, \hat{v}_0, \hat f)$, respectively. Then the stability estimate
\[
\begin{split}
  &\|v(t)-\hat{v}(t)\|_{H}^{2}+\frac{1}{\lambda}\left\| \left(K_{u_0}v\right)(t)-\left(K_{\hat u_0}\hat{v}\right)(t) \right\|_{B}^{2}+\int_0^t \left\|\left(K_{u_0}v\right)(s)-\left(K_{\hat u_0}\hat{v}\right)(s)\right\|_{B}^{2}\diff s\\
  &\leq c\left(\|v_0-\hat v_0\|_H^2 + \|u_0-\hat u_0\|_{V_B}^2 + \|f-\hat f\|^2_{L^1(0,T;H)}\right)
\end{split}
\]
holds for all $t\in [0,T]$, where the index in the notation $K_{u_0}$, $K_{\hat u_0}$ denotes the dependence of the integral operator on $u_0$, $\hat u_0$, respectively.
\item Let $f-\hat f\in L^{p'}(0,T;V_A')$. Assume in addition that $A$ is uniformly monotone in the sense that there is $\mu>0$ such that
\[\langle Av - Aw, v - w \rangle \ge \mu \|v-w\|_{V_A}^p\]
for all $v,w\in V_A$. Let $v$, $\hat  v$ be solutions to (\ref{abstract1}) with data $({u}_0, {v}_0,  f)$ and $(\hat{u}_0, \hat{v}_0, \hat f)$, respectively. Then the stability estimate
\[
\begin{split}
  &\|v(t)-\hat{v}(t)\|_{H}^{2}+{\mu}\int_0^t\|v(s)-\hat{v}(s)\|_{V_A}^{p}ds+\frac{1}{\lambda}\left\| \left(K_{u_0}v\right)(t)-\left(K_{\hat u_0}\hat{v}\right)(t)\right\|_{B}^{2}\\
  &\hspace{0.5cm}+\int_0^t \left\| \left(K_{u_0}v\right)(s)-\left(K_{\hat u_0}\hat{v}\right)(s)\right\|_{B}^{2}\diff s\leq c\Big(\|v_0-\hat v_0\|_H^2 + \|u_0-\hat u_0\|_{V_B}^2 + \|f-\hat f\|^2_{L^{p'}(0,T;V_A')}\Big)
\end{split}
\]
holds for all $t\in [0,T]$.
\end{enumerate}
\end{theorem}
\begin{proof}
 Let $v$, $\hat v$ be solutions for  $({u}_0, {v}_0,  f)$,  $(\hat{u}_0, \hat{v}_0, \hat f)$, respectively. The difference of both equations \eqref{abstract1a} then reads
\begin{equation}\label{EqStab1}
(v-\hat v)' + Av - A\hat v + B\left(K_{u_0} v-K_{\hat u_0}\hat v\right)= f-\hat f.
\end{equation}
Lemma \ref{PartInt} implies for all $t\in [0,T]$
\[
 \begin{split}
  &\int^t_0 \Big\langle (v-\hat v)'(s) +B\left(\left(K_{u_0} v\right)(s)-\left(K_{\hat u_0}\hat v\right)(s)\right) , (v-\hat v)(s)\Big\rangle \diff s \\
  &= \frac{1}{2}\|v(t)-\hat v(t)\|^2_H - \frac{1}{2}\|v_0-\hat v_0\|^2_H + \frac{1}{2\lambda}\left\|\left(K_{u_0} v\right)(t)-\left(K_{\hat u_0}\hat v\right)(t)\right\|^2_{B} -\frac{1}{2\lambda}\|u_0-\hat u_0\|^2_{B}  \\
  &\phantom{MMM}- \int_0^t \left\langle B\left(\left(K_{u_0} v\right)(s)-\left(K_{\hat u_0}\hat v\right)(s)\right), u_0-\hat u_0\right\rangle \diff s + \int_0^t \left\|\left(K_{u_0} v\right)(s)-\left(K_{\hat u_0}\hat v\right)(s)\right\|^2_{B}\diff s.
\end{split}
\]
We test in equation (\ref{EqStab1}) with $v-\hat v$. To prove (i), the monotonicity of $A$ and Young's inequality imply
\[
 \begin{split}
  &\frac{1}{2}\|v(t)-\hat v(t)\|^2_H  + \frac{1}{2\lambda}\left\|\left(K_{u_0} v\right)(t)-\left(K_{\hat u_0}\hat v\right)(t)\right\|^2_{B}    + \int_0^t \left\|\left(K_{u_0} v\right)(s)-\left(K_{\hat u_0}\hat v\right)(s)\right\|^2_{B}\diff s  \\
  &\leq\frac{1}{2}\|v_0-\hat v_0\|^2_H + \frac{1}{2\lambda}\|u_0-\hat u_0\|^2_{B} \\
  &\phantom{MMM}+ \int_0^t \left\langle B\left(\left(K_{u_0} v\right)(s)-\left(K_{\hat u_0}\hat v\right)(s)\right), u_0-\hat u_0\right\rangle \diff s + \int_0^t \Big\langle (f-\hat f)(s), (v-\hat v)(s)\Big\rangle \diff s\\
  &\leq c\left(\|v_0-\hat v_0\|^2_H +\|u_0-\hat u_0\|^2_{B} + \|f-\hat f\|_{L^1(0,T;H)}^2\right)\\
  &\phantom{MMM}+ \frac{1}{2}\int_0^t \left\|\left(K_{u_0} v\right)(s)-\left(K_{\hat u_0}\hat v\right)(s)\right\|^2_B \diff s +\frac{1}{4}\|v-\hat v\|_{L^\infty(0,T;H)}^2.
\end{split}
\]
An analogous argument to the one used in the proof of Theorem \ref{aprioriestimate} yields the first statement.
The second one follows analogously from the uniform monotonicity of $A$ and
\[
\int_0^t \Big\langle (f-\hat f)(s), (v-\hat v)(s)\Big\rangle \diff s\leq \|f-\hat f\|_{L^{p'}(0,T;V_A')} \int_0^t\|(v-\hat v)(s)\|_{V_A}^p \diff s.
\]
\qed
\end{proof}

As usually, the stability estimates directly provide a uniqueness result, which we formulate in the following corollary.

\begin{cor}[Uniqueness]\label{ThmUniqueness} Under the assumptions of Theorem \ref{MainThm}, the solution to problem (\ref{abstract1}) is unique. Furthermore, the whole sequences $\{v_{\ell}\}$ and $\{\hat{v}_{\ell}\}$ of piecewise constant and of piecewise affine-linear prolongations of solutions to the discrete problem (\ref{scheme}) converge to the solution in the sense stated in Theorem~\ref{MainThm}.
\end{cor}

\begin{proof}
We consider $u_0=\hat u_0\in V_B$, $v_0=\hat v_0\in H$ and $f=\hat f\in L^{p'}(0,T;V_A')+L^1(0,T;H)$ with corresponding solutions $v$ and $\hat v$. As $f-\hat f=0 \in L^1(0,T;H)$, the first assumption of Theorem \ref{ThmStability} is fulfilled, which shows
%

\[
\|v(t)-\hat{v}(t)\|_{H}^{2}+\frac{1}{\lambda}\left\|(Kv)(t)-(K\hat{v})(t)\right\|_{B}^{2}+\int_0^t \|(Kv)(s)-(K\hat{v})(s)\|_{B}^{2}\diff s
\leq0
\]
for all $t\in [0,T]$. This proves the uniqueness. Convergence of the whole sequences then follows as usual
by contradiction. \qed
\end{proof}

Finally, we aim to derive Lipschitz-type dependence on the kernel parameter $1/\lambda$, that is Lipschitz dependence on the average relaxation-time.

\begin{theorem}[Perturbation of $\lambda$]
Let the assumptions of Theorem \ref{MainThm} be fulfilled and let $v_{0}\in V$. For $\lambda>0$, let $K_{\lambda}$ denote the operator related to the kernel $k_{\lambda}(t)=\lambda e^{-\lambda t}$ and let $v_{\lambda}\in L^p(0,T;V_A)\cap L^\infty(0,T;H)$ denote the corresponding solution to problem (\ref{abstract1}). Moreover, assume that $v_\lambda\in L^2(0,T;V_B)$.
Then for all $\mu>0$ and corresponding solutions $v_\mu\in L^p(0,T;V_A)\cap L^\infty(0,T;H)$, there holds
\[\begin{split}
&\|(v_{\lambda}-v_{\mu})(t)\|_H^2  + \int_{0}^{t} \|(K_{\lambda}v_{\lambda}-K_{\mu}v_{\mu})    (s)\|_B^2 \diff s
\leq \frac{\lambda^2}{2}(1+\lambda^2 T^2)\left|\frac{1}{\lambda}-\frac{1}{\mu}\right|^2\|v_\lambda\|_{L^2(0,T;B)}^2
\end{split}\]
for almost all $t$. 

\end{theorem}

\begin{proof} Taking the difference of the equations (\ref{abstract1a})
for $v_{\lambda}$ and $v_{\mu}$, respectively, and testing
with $v_{\lambda}-v_{\mu}$ leads to
\begin{equation}\begin{split}
&\int_{0}^{t}\Big\langle(v_{\lambda}-v_{\mu})'(s)+B((K_{\lambda}v_{\lambda})(s)-(K_{\mu}v_{\mu})(s)),(v_{\lambda}-v_{\mu})(s)\Big\rangle \diff s\\&=-\int_{0}^{t}\Big\langle Av_{\lambda}(s)-Av_{\mu}(s),(v_{\lambda}-v_{\mu})(s)\Big\rangle \diff s\leq0.\label{Stabil1}
\end{split}
\end{equation}

To deal with the integral on the left-hand side, we need to prove an integration-by-parts formula similar to (\ref{PartInt2}). Unfortunately, due to the fact that we have two different parameters $\lambda$ and $\mu$ instead of one, we are not able to apply the same method as in the proof of Lemma \ref{PartInt}. Therefore, we use another method of proof using the centered Steklov average.\footnote{We could have also used this method to prove Lemma \ref{PartInt}, but it only provides the integration-by-parts formula to hold on $(\alpha,\beta)$ for almost all $\alpha,\beta\in (0,T)$.}

In order to prove the integration-by-parts formula on $(0,t)$, we first have to extend the functions considered to the interval $(-\eta,0)$. We then prove the formula on an arbitrary interval $(\alpha,\beta)$ with $-\eta+h_0<\alpha<\beta<T-h_0$. At the end, instead of a fixed $\alpha$, we consider a sequence $\{\alpha_k\}$ with $\alpha_k<0$ and $\alpha_k\to 0$ as $k\to \infty$. If we considered a sequence of positive $\alpha_k$, we would not be able to identify the limit properly.

We fix $h_0>0$. The centered Steklov average for a function $z\in L^p(-\eta,T;X)$, $X$ being an arbitrary Banach space, is defined by
\begin{equation}\label{Steklov}
(S_h z)(t):= \frac{1}{2h} \int_{t-h}^{t+h} z(s)\diff s ,
\end{equation}
where $0<h<h_0$, $t\in [-\eta+h_0, T-h_0]$.
It is well-known (see Diestel and Uhl \cite[Thm.~9, p.~49]{DiestelUhl}) that $ (S_hz)(t)\to z(t)$ in $X$ for almost all $t\in (-\eta+h_0,T-h_0)$. In addition, it is easy to show that $S_h z\in L^p(-\eta+h_0,T-h_0;X)$ as well as, using Fubini's theorem, $\|S_hz\|_{L^p(-\eta+h_0,T-h_0;X)}\leq \|z\|_{L^p(-\eta,T;X)}$ and thus $S_hz\to z$ in $L^p(-\eta+h_0,T-h_0;X)$.

First, note that both $v_{\lambda}$ and $v_{\mu}$ are in $\mathscr{C}_{w}([0,T];H)$ and take the same value $v_{0}$ at $t=0$. Hence, the difference $v_{\lambda}-v_{\mu}$ fulfills the initial condition $(v_{\lambda}-v_{\mu})(0)=0$. We assume both $v_\lambda$ and $v_\mu$ to be extended by $v_0\in V$ for $t<0$ and thus the difference to be extended by $0$ for $t<0$. For fixed $\eta<0$, we conclude that $v_{\lambda}-v_{\mu}\in L^p(-\eta,T;V_A)\cap \mathscr{C}_{w}([-\eta,T];H)$ and similarly $K_{\lambda}v_{\lambda}-K_{\mu}v_{\mu}\in  \mathscr{C}_{w}([-\eta,T];V_{B})$. The relation \eqref{GleichungK} still holds for both $v_\lambda$ and $v_\mu$ for almost all $t\in(-\eta,0)$.

As $v_{\lambda}-v_{\mu}\in L^p(-\eta,T;V_A)\cap \mathscr{C}_{w}([-\eta,T];H)$, it is easy to show that $S_h(v_{\lambda}-v_{\mu})$ is bounded in $L^p(-\eta +h_0,T-h_0;V_A)\cap L^\infty(-\eta+h_0,T-h_0;H)$. Therefore, there exists a subsequence, again denoted by $h$, such that $S_h (v_{\lambda}-v_{\mu}) \wsconv v_{\lambda}-v_{\mu}$ in $L^p(-\eta+h_0,T-h_0;V_A)\cap L^\infty(-\eta+h_0,T-h_0;H)$. In view of (\ref{GleichungK}), we have
\begin{equation} \label{Equation_S_h}
  \begin{split}
    &(S_h  (v_{\lambda}-v_{\mu}))(t)= \frac{1}{2h} \left(\left(\frac{1}{\lambda} K_\lambda v_\lambda -\frac{1}{\mu} K_\mu v_\mu\right)(t+h)-\left(\frac{1}{\lambda} K_\lambda v_\lambda -\frac{1}{\mu} K_\mu v_\mu\right)(t-h)\right) \\
    &\hspace{8cm} \phantom{\int}+(S_h K_\lambda v_\lambda) (t) -(S_h K_\mu v_\mu)(t)
  \end{split}
\end{equation}
such that $K_\lambda v_\lambda, K_\mu v_\mu\in \mathscr{C}_w([-\eta,T];V_B)$ provide $S_h(v_\lambda-v_\mu)\in \mathscr{C}_w([-\eta+h_0,T-h_0];V_B)$. Thus -- in contrast to the difference $v_\lambda-v_\mu$ itself -- the Steklov average $S_h (v_\lambda-v_\mu)$ is an element of $L^p(-\eta+h_0,T-h_0;V_A)\cap \mathscr{C}_w([-\eta+h_0,T-h_0];V_B)$, which provides the regularity needed.

We find for all $\alpha, \beta$ with $-\eta+h_0<\alpha<\beta<T-h_0$
\begin{equation} \label{split_S_h}
  \begin{split}
    &\int_{\alpha}^{\beta}\Big\langle(v_{\lambda}-v_{\mu})'(t)+B((K_{\lambda}v_{\lambda})(t)-(K_{\mu}v_{\mu})(t)),(v_{\lambda}-v_{\mu})(t)\Big\rangle \diff t \\
    & =\int_{\alpha}^{\beta}\Big\langle(v_{\lambda}-v_{\mu})'(t)+B((K_{\lambda}v_{\lambda})(t)-(K_{\mu}v_{\mu})(t)),(v_{\lambda}-v_{\mu})(t)-(S_{h}(v_{\lambda}-v_{\mu}))(t)\Big\rangle \diff t\\
    & \phantom{MMM}+\int_{\alpha}^{\beta}\Big\langle (v_{\lambda}-v_{\mu})'(t)+B((K_{\lambda}v_{\lambda})(t)-(K_{\mu}v_{\mu})(t))\\&\phantom{MMMMMMMM}- \left(S_h\left((v_{\lambda}-v_{\mu})'+B (K_{\lambda}v_{\lambda}-K_{\mu}v_{\mu})\right)\right)(t), (S_{h}(v_{\lambda}-v_{\mu}))(t)\Big\rangle \diff t\\
    & \phantom{MMM}+\int_{\alpha}^{\beta}\Big\langle \left(S_h (v_{\lambda}-v_{\mu})'\right)(t)+B(S_h(K_{\lambda}v_{\lambda}-K_{\mu}v_{\mu}))(t), (S_{h}(v_{\lambda}-v_{\mu}))(t)\Big\rangle \diff t.
  \end{split}
\end{equation}
The first term on the right-hand side tends to zero as $h\to 0$ since $$(v_\lambda-v_\mu)'+B(K_\lambda v_\lambda-K_\mu v_\mu)=-(Av_\lambda-Av_\mu)\in L^{p'}(-\eta+h_0,T-h_0; V_A')+L^1(-\eta+h_0,T-h_0;H)$$ and $$S_h (v_\lambda-v_\mu)\wsconv v_\lambda-v_\mu$$ in  $L^p(-\eta+h_0,T-h_0; V_A)\cap L^\infty(-\eta+h_0,T-h_0;H)$ as $h\to 0$, the second one since $\{S_h(v_\lambda-v_\mu)\}$ is bounded in $L^p(-\eta+h_0,T-h_0;V_A)\cap L^\infty(-\eta+h_0,T-h_0;H)$ and $$ S_h \left((v_\lambda-v_\mu)' + B(K_\lambda v_\lambda-K_\mu v_\mu)\right)\to (v_\lambda-v_\mu)'+B(K_\lambda v_\lambda-K_\mu v_\mu)$$ in $L^{p'}(-\eta+h_0,T-h_0;V_A')+L^1(-\eta+h_0,T-h_0;H)$ as $h\to 0$. The latter follows from the linearity of $S_h$.

%
%

Due to the regularity properties of $S_h(v_\lambda-v_\mu)$, we are allowed to split the last integrand in $\eqref{split_S_h}$. We use \[v_{\lambda}-v_{\mu} = \frac{1}{\lambda}(K_{\lambda}v_{\lambda})'-\frac{1}{\mu} (K_{\mu}v_{\mu})' +  K_{\lambda}v_{\lambda}- K_{\mu}v_{\mu}\] as well as $S_h (z')= (S_h z)'$ for any integrable function $z$ on $(\alpha, \beta)\subset (-\eta+h_0,T-h_0)$. We rewrite the third term as
\[\begin{split}
  &\int_{\alpha}^{\beta}\Big\langle \left(S_h (v_{\lambda}-v_{\mu})'\right)(t)+B(S_h(K_{\lambda}v_{\lambda}-K_{\mu}v_{\mu}))(t),(S_{h}(v_{\lambda}-v_{\mu}))(t)\Big\rangle \diff t\\
  &= \int_{\alpha}^{\beta}\Big\langle \left(S_h (v_{\lambda}-v_{\mu})'\right)(t), (S_{h}(v_{\lambda}-v_{\mu}))(t)\Big\rangle \diff t\\
  &\hspace{1cm}+\int_{\alpha}^{\beta}\Big\langle B(S_h(K_{\lambda}v_{\lambda}-K_{\mu}v_{\mu}))(t), (S_{h} (K_{\lambda}v_{\lambda} -K_{\mu}v_{\mu}))(t) \Big\rangle \diff t\\
  &\hspace{2cm}+\int_{\alpha}^{\beta}\Big\langle B(S_h(K_{\lambda}v_{\lambda}-K_{\mu}v_{\mu}))(t), \frac{1}{\lambda}\left(S_{h} (K_{\lambda}v_{\lambda})'\right)(t)-\frac{1}{\mu}\left(S_{h}(K_{\mu}v_{\mu})'\right)(t)\Big\rangle \diff t\\
  &= \frac{1}{2}\|(S_h (v_{\lambda}-v_{\mu}))(\beta)\|_H^2 -\frac{1}{2} \|(S_h (v_{\lambda}-v_{\mu}))(\alpha)\|_H^2 + \int_{\alpha}^{\beta} \|(S_h   (K_{\lambda}v_{\lambda}-K_{\mu}v_{\mu}) ) (t)\|_B^2 \diff t\\
  &\hspace{1cm}+ \frac{1}{\mu}\int_{\alpha}^{\beta}\Big\langle B(S_h(K_{\lambda}v_{\lambda}-K_{\mu}v_{\mu}))(t), \left(S_{h} (K_{\lambda}v_{\lambda}-K_{\mu}v_{\mu})'\right)(t)\Big\rangle \diff t\\
  &\hspace{2cm}+ \Big(\frac{1}{\lambda}-\frac{1}{\mu}\Big)\int_{\alpha}^{\beta}\Big\langle B(S_h(K_{\lambda}v_{\lambda}-K_{\mu}v_{\mu}))(t), \left(S_{h} (K_{\lambda}v_{\lambda})'\right)(t)\Big\rangle \diff t\\
  &\geq \frac{1}{2}\|(S_h (v_{\lambda}-v_{\mu}))(\beta)\|_H^2 -\frac{1}{2} \|(S_h (v_{\lambda}-v_{\mu}))(\alpha)\|_H^2 + \int_{\alpha}^{\beta} \|(S_h   (K_{\lambda}v_{\lambda}-K_{\mu}v_{\mu}) ) (t)\|_B^2 \diff t\\
  &\hspace{1cm}+ \frac{1}{2\mu}\|(S_h(K_{\lambda}v_{\lambda}-K_{\mu}v_{\mu}))(\beta)\|_B^2 - \frac{1}{2\mu}\| (S_{h} (K_{\lambda}v_{\lambda}-K_{\mu}v_{\mu}))(\alpha)\|_B^2\\
  &\hspace{2cm}- \Big|\frac{1}{\lambda}-\frac{1}{\mu}\Big| \int_\alpha^\beta\| (S_h(K_{\lambda}v_{\lambda}-K_{\mu}v_{\mu}))(t)\|_B \;\!\| \left(S_{h} (K_{\lambda}v_{\lambda})'\right)(t)\|_B \diff t.\\
\end{split}\]
Taking the limit $h\to \infty$ and employing (\ref{Stabil1}) results in

\begin{equation}\label{Stabil2}
  \begin{split}
    &\frac{1}{2}\|(v_{\lambda}-v_{\mu})(\beta)\|_H^2  + \int_{\alpha}^{\beta} \|(K_{\lambda}v_{\lambda}-K_{\mu}v_{\mu})    (t)\|_B^2 \diff t + \frac{1}{2\mu}\|(K_{\lambda}v_{\lambda}-K_{\mu}v_{\mu})(\beta)\|_B^2 \\
    &\leq \frac{1}{2} \|(v_{\lambda}-v_{\mu})(\alpha)\|_H^2 +  \frac{1}{2\mu}\| (K_{\lambda}v_{\lambda}-K_{\mu}v_{\mu})(\alpha)\|_B^2\\
    &\hspace{1cm}+ \Big|\frac{1}{\lambda}-\frac{1}{\mu}\Big| \int_\alpha^\beta\| (K_{\lambda}v_{\lambda}-K_{\mu}v_{\mu})(t)\|_B\;\! \|  (K_{\lambda}v_{\lambda})'(t)\|_B \diff t
  \end{split}
\end{equation}
for almost all $\alpha,\beta$, namely for those that are Lebesgue points of $v_{\lambda}-v_{\mu}$ and $K_{\lambda}v_{\lambda}-K_{\mu}v_{\mu}$. Finally, we consider a sequence $\{\alpha_k\}$ with $\alpha_k<0$ and $\alpha_k\to 0$ as $k\to \infty$. We remark that $v_{\lambda}-v_{\mu}\in \mathscr{C}_w([-\eta,T];H)$ and $K_{\lambda}v_{\lambda}-K_{\mu}v_{\mu}\in \mathscr{C}_w([-\eta,T];V_B)$. Due to the choice of the extension for $t<0$, equation (\ref{Stabil2}) shows for $\beta=t$ that
\[
  \begin{split}
    &\frac{1}{2}\|(v_{\lambda}-v_{\mu})(t)\|_H^2  + \int_{0}^{t} \|(K_{\lambda}v_{\lambda}-K_{\mu}v_{\mu})    (s)\|_B^2 \diff s + \frac{1}{2\mu}\|(K_{\lambda}v_{\lambda}-K_{\mu}v_{\mu})(t)\|_B^2 \\
    &\leq \Big|\frac{1}{\lambda}-\frac{1}{\mu}\Big| \int_0^t\| (K_{\lambda}v_{\lambda}-K_{\mu}v_{\mu})(s)\|_B \;\!\|  (K_{\lambda}v_{\lambda})'(s)\|_B \diff s\\
  \end{split}
\]
for almost all $t\in (0,T)$ and thus
\[\begin{split}
  &\frac{1}{2}\|(v_{\lambda}-v_{\mu})(t)\|_H^2  + \frac{1}{2}\int_{0}^{t} \|(K_{\lambda}v_{\lambda}-K_{\mu}v_{\mu})    (s)\|_B^2 \diff s + \frac{1}{2\mu}\|(K_{\lambda}v_{\lambda}-K_{\mu}v_{\mu})(t)\|_B^2 \\
  &\leq \frac{1}{2}\Big|\frac{1}{\lambda}-\frac{1}{\mu}\Big|^2 \|(K_{\lambda}v_{\lambda})'\|_{L^2(0,T;B)}^2\\
  &\leq \frac{\lambda^2}{2}\Big|\frac{1}{\lambda}-\frac{1}{\mu}\Big|^2\left( \|v_\lambda\|_{L^2(0,T;B)}^2 + \|K_{\lambda}v_{\lambda}-u_0\|_{L^2(0,T;B)}^2\right).
\end{split}\]
It is easy to show that $\|K_{\lambda}v_{\lambda}-u_0\|_{L^2(0,T;B)}^2\leq \lambda^2 T^2\|v_\lambda\|_{L^2(0,T;B)}^2$, which proves the assertion.
\qed

\end{proof}

\end{document}